\documentclass[12pt]{article}
\usepackage{e-jc}

\usepackage{amsthm,amsmath,amssymb}

\usepackage{graphicx}

\usepackage[colorlinks=true,citecolor=black,linkcolor=black,urlcolor=blue]{hyperref}

\theoremstyle{plain}
\newtheorem{theorem}{Theorem}
\newtheorem{lemma}[theorem]{Lemma}

\theoremstyle{definition}

\theoremstyle{remark}

\newcommand{\M}{\operatorname{M}}

\newcommand{\T}{\operatorname{T}}
\title{A simple proof for the number of tilings of quartered Aztec diamonds}

\author{Tri Lai\\
\small Department of Mathematics\\[-0.8ex]
\small Indiana University\\[-0.8ex]
\small Bloomington, IN 47405\\
\small\tt tmlai@indiana.edu
}

\date{\small Mathematics Subject Classifications: 05A15, 05C70}

\begin{document}

\maketitle

% E-JC papers must include an abstract. The abstract should consist of a
% succinct statement of background followed by a listing of the
% principal new results that are to be found in the paper. The abstract
% should be informative, clear, and as complete as possible. Phrases
% like "we investigate..." or "we study..." should be kept to a minimum
% in favor of "we prove that..."  or "we show that...".  Do not
% include equation numbers, unexpanded citations (such as "[23]"), or
% any other references to things in the paper that are not defined in
% the abstract. The abstract will be distributed without the rest of the
% paper so it must be entirely self-contained.

\begin{abstract}
 We get four quartered Aztec diamonds by dividing an Aztec diamond region by two zigzag cuts passing its center. W. Jockusch and J. Propp (in an unpublished work) found that the number of tilings of quartered Aztec diamonds is given by simple product formulas. In this paper we present a simple proof for this result.

  % keywords are optional
\end{abstract}

\section{Introduction}
In this paper a (lattice) \textit{region} is a connected union of unit squares in the square lattice. A domino is the union of two unit squares that share an edge. A (domino) \textit{tiling} of a region $R$ is a covering of $R$ by dominos such that there are no gaps or overlaps. Denote by $\T(R)$ the number of tilings of the region $R$.

 The \textit{Aztec diamond} of order $n$ is defined to be the union of all the unit squares with integral corners $(x,y)$ satisfying $|x|+|y|\leq n+1$. The Aztec diamond of order $8$ is shown in Figure \ref{QA2}(a). In \cite{Elkies} it was shown that the number of tilings of the Aztec diamond of order $n$ is $2^{n(n+1)}$.

\begin{figure}\centering
\includegraphics[width=10cm]{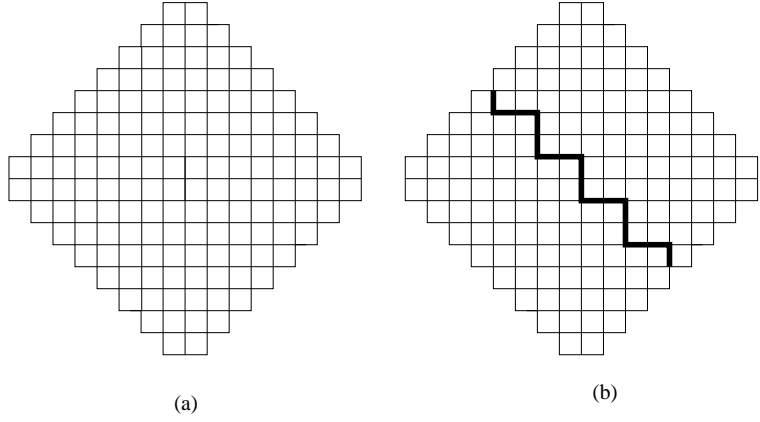}
\caption{ The Aztec diamond region of order $8$, and its division into two congruent parts.}
\label{QA2}
\end{figure}

Next, we consider three related families of regions name \textit{quartered Aztec diamonds}, that were introduced by Jockusch and Propp \cite{JP}. Divide the Aztec diamond of order $n$ into two congruent parts by a zigzag cut with 2-unit steps (see Figure \ref{QA2}(b) for an example with $n=8$). By superimposing two such zigzag cuts that pass the center of the Aztec diamond we partition the region into four parts, called quartered Aztec diamonds. Up to symmetry, there are essentially two different ways we can superimpose the two cuts. For one of them, we obtained a fourfold rotational symmetric pattern, and four resulting parts are congruent (see Figure \ref{QA}(a)). Denote by $R(n)$ these quartered Aztec diamonds. For the other, the obtained pattern has Klein 4-group reflection symmetry and there are two different kinds of quartered Aztec diamonds (see Figure \ref{QA} (b)); they are called \textit{abutting} and \textit{non-abutting} quartered Aztec diamonds. Denote by $K_{a}(n)$ and $K_{na}(n)$ the abutting and non-abutting quartered Aztec diamonds of order $n$, respectively.

\begin{figure}\centering
\includegraphics[width=10cm]{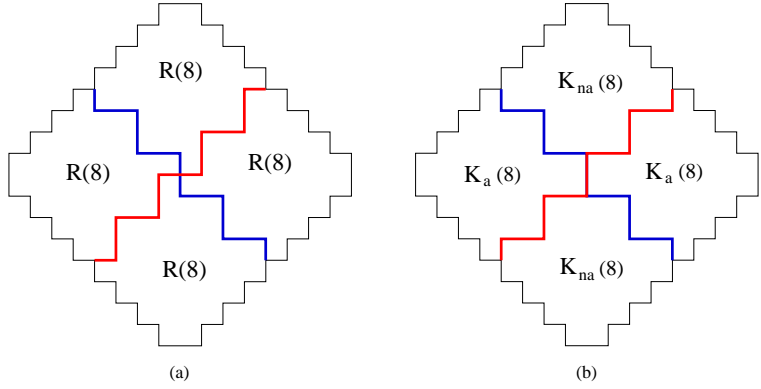}
\caption{Three kinds of quartered Aztec diamonds of order $8$.}
\label{QA}
\end{figure}

The number of tilings of a region three kinds of quartered Aztec diamond can be obtained by the theorem stated below (\cite{JP}, Theorem 1)

\begin{theorem}\label{main}

\begin{equation}\label{main1}
\T(R(4n+1))=\T(R(4n+2))=0
\end{equation}

\begin{equation}\label{main2}
\T(R(4n))=2^n\T(R(4n-1))=2^{n(3n-1)/2}\prod_{1\leq i<j\leq n}\frac{2i+2j-1}{i+j-1}
\end{equation}

%\begin{equation}\label{main2}
%M(R(4n))=2^nM(R(4n-1))=2^{n(3n-3)/2}\prod_{1\leq i,j\leq n}(2n-1+2j-2i)\prod_{1\leq i<j\leq n}\frac{(i+j-1)(j-i)}{2i+2j-1}
%\end{equation}

\begin{equation}\label{main3}
\T(K_{a}(4n-2))=\T(K_a(4n))=2^{n(3n-1)/2}\prod_{1\leq i<j\leq n}\frac{2i+2j-3}{i+j-1}
\end{equation}

\begin{equation}\label{main4}
\T(K_{a}(4n-1))=\T(K_a(4n+1))=2^{n(3n-3)/2}\prod_{1\leq i\leq j\leq n}\frac{2i+2j-1}{i+j-1}
\end{equation}

\begin{equation}\label{main5}
\T(K_{na}(4n))=\T(K_{na}(4n+2))=2^{n(3n-1)/2}\prod_{1\leq i\leq j\leq n}\frac{2i+2j-1}{i+j-1}
\end{equation}

\begin{equation}\label{main6}
\T(K_{na}(4n-3))=\T(K_{na}(4n-1))=2^{n(3n-3)/2}\prod_{1\leq i<j\leq n}\frac{2i+2j-3}{i+j-1}
\end{equation}

\end{theorem}

In \cite{JP}, Juckusch and Propp presented a proof for Theorem 1 by investigating properties of ``\emph{antisymmetric monotone triangles}". We will prove Theorem \ref{main} by a visual way in the next section

\section{Proof of Theorem \ref{main}}
We have 4  recurrences  that were proved by M. Ciucu  in\cite{Ciucu2},  Theorem 4.1.

\begin{lemma}\label{lem1} For all $n\geq 1$ we have
\begin{equation}\label{eq1}
\T(R_{4n})=2^n\T(R_{4n-1})
\end{equation}
\begin{equation}\label{eq2}
\T(K_{na}(4n+1))=2^n\T(K_{na}(4n))
\end{equation}
\begin{equation}\label{eq3}
\T(K_{na}(4n))=2^n\T(K_{a}(4n-1))
\end{equation}
\begin{equation}\label{eq4}
\T(K_{a}(4n-2))=2^n\T(K_{na}(4n-3))
\end{equation}
\end{lemma}

The \textit{dual graph} of a region $R$ is the graph whose vertices are unit square in $R$ and whose edges connect precisely two unit squares sharing an edge. A \textit{perfect matching} of a graph $G$ is a collection of edges such that each vertex of $G$ is adjacent to exactly one selected edge. Denote by $\M(G)$ the number of perfect matchings of $G$. By a well-known bijection between tilings of a region and perfect matchings of its dual graph, we enumerate perfect matchings of the dual graph of a region rather than enumerating its tilings directly. Since we are considering only regions in the square lattice, one can view the dual graphs of those regions as subgraphs of the infinite square grid $\mathbb{Z}^2$.

An edge in a graph $G$ is called a \textit{forced edge}, if it is in every perfect matching of $G$. One can remove some forced edges from a graph to get a new graph with the same number of perfect matchings. We have the following lemma by considering forced edges in the dual graphs of quartered Aztec diamonds.

\begin{figure}\centering
\includegraphics[width=10cm]{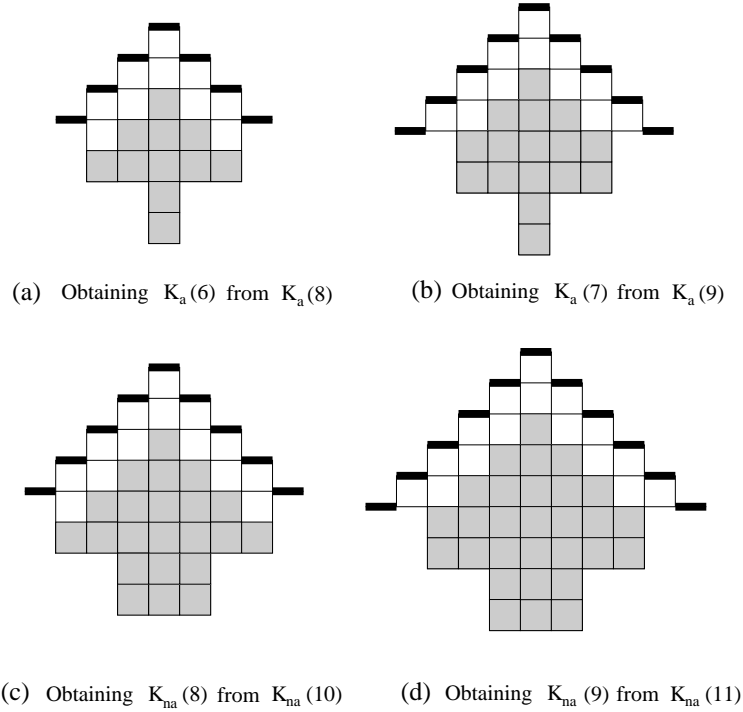}
\caption{Illustrating the proof of Lemma \ref{lem2}.}
\label{force1}
\end{figure}

\begin{lemma}\label{lem2} For any $n\geq 1$
\begin{equation}\label{eq5}
\T(K_a(4n-2))=\T(K_{a}(4n))
\end{equation}
\begin{equation}\label{eq6}
\T(K_{a}(4n-1))=\T(K_{a}(4n+1))
\end{equation}
\begin{equation}\label{eq7}
\T(K_{na}(4n))=\T(K_{na}(4n+2))
\end{equation}
\begin{equation}\label{eq8}
\T(K_{na}(4n+1))=\T(K_{na}(4n+3))
\end{equation}
\end{lemma}

\begin{proof}
Instead of comparing the numbers of tilings of the regions, we compare the numbers of perfect matchings of their dual graphs. In each of the four equalities, the dual graph of the region on the left is obtained from the dual graph of the region on the right by removing forced edges. The proofs of   (\ref{eq5})-(\ref{eq8}) are illustrated by Figures \ref{force1} (a)-(d), respectively. In these figures, the forced edges are represented by the bold edges, and the dual graph of the region on the left of each equality is represented by the graph consisting of shaded unit squares.
\end{proof}

\begin{figure}\centering
\includegraphics[width=12cm]{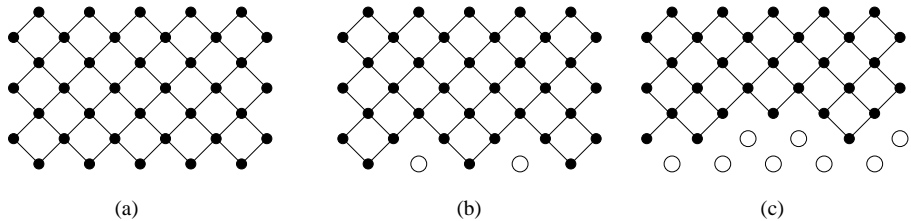}
\caption{The Aztec rectangle and two holey Aztec rectangles of order $3\times 5$.}
\label{HoleyAR}
\end{figure}

Next, we consider a well-known family of graphs as follows.  Consider a $(2m+1)\times (2n+1)$ rectangular chessboard and suppose that the corners are black. The $m\times n$ \textit{Aztec rectangle} is the graph whose vertices are the white square and whose edges connect precisely those pairs of white squares  that are diagonally adjacent (see Figure \ref{HoleyAR}(a) for an example with $m=3$ and $n=5$). %The $m\times n$ Aztec rectangle does not have a tiling, except for the case when $m=n$. In the latter case the graph is exactly the dual graph of the Aztec diamond region of order $n$.
We are interested in the the number of perfect matchings of two families of \textit{holey Aztec rectangles} as follows.

\begin{lemma}[see \cite{Ciucu1}, (4.4); or \cite{Krat}, Lemma 1] \label{lem4}The number of perfect matchings of a $m\times n$ Aztec rectangle, where all the vertices in the bottom-most row, except for the $a_1$-st, the $a_2$-nd, $\dots$, and the $a_m$-th vertex, have been removed (see Figure \ref{HoleyAR}(b) for an example with $m=3$, $n=5$, $a_1=1$, $a_2=3$, $a_3=5$), equals
\begin{equation}
2^{m(m+1)/2}\prod_{1\leq i<j\leq m}\frac{a_j-a_i}{j-i}
\end{equation}
\end{lemma}

Next, we consider a variant of the lemma above (see \cite{Gessel}, Lemma 2; or \cite{Krat}, Lemma 2).

\begin{lemma}\label{lem5} The number of perfect matchings of a $m\times n$ Aztec rectangle, where all the vertices in the bottom-most row have been removed, and where the $a_1$-st, the $a_2$-nd, $\dots$, and the $a_m$-th vertex, have been removed from the resulting graph (see Figure \ref{HoleyAR}(c),  for and example with $m=3$, $n=5$, $a_1=3$, $a_2=4$,$a_3=6$), equals
\begin{equation}
2^{m(m-1)/2}\prod_{1\leq i<j\leq m}\frac{a_j-a_i}{j-i}
\end{equation}
\end{lemma}

Denote by $AR_{m,n}(\{a_1,\dotsc,a_k\})$ and $\overline{AR}_{m,n}(\{a_1,\dotsc,a_k\})$ the graphs in Lemmas \ref{lem4} and \ref{lem5}, respectively.

Let $G$ be a connected subgraph of $\mathbb{Z}^2$ symmetric about a diagonal lattice $l$. Assume all the vertices of $G$ on $l$ are consecutive lattice points on that line. Go along the line $l$ from left to right, and alternate between deleting the edges of $G$ that touch them from below, and deleting the edges of $G$ that touch them from above (see Figure \ref{factor} for an example). Let $G^+$ and $G^-$ be the connected components of the resulting graph that are above and below $l$. It is easy to see that the number of vertices of $G$ on $l$ must be even if $G$ has perfect matchings, let $w(G)$ be half of this number. Then Ciucu's Factorization Theorem \cite{Ciucu1} implies that
\begin{equation}\label{factoreq}
\M(G)=2^{w(G)}\M(G^+)\M(G^-)
\end{equation}

\begin{figure}\centering
\includegraphics[width=10cm]{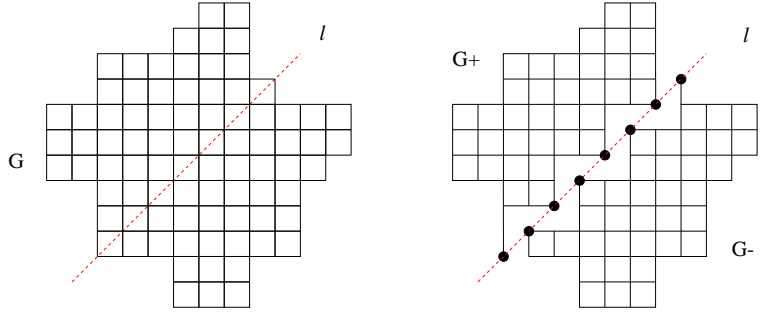}
\caption{A symmetric graph $G$, and two graphs $G^+$ and $G^-$ after the cutting procedure.}
\label{factor}
\end{figure}

By applying the Factorization Theorem we get new properties of quartered Aztec diamonds stated in the lemma below.

\begin{lemma}\label{lem3} For $n\geq 1$ we have

\begin{equation}\label{eq9}
 \M(AR_{2n,4n}(\mathcal{B}_n))=2^{n}\T(R(4n))\T(K_{a}(4n)),
\end{equation}

\begin{equation}\label{eq10}
\M(AR_{2n,4n}(\mathcal{A}_n))=2^{n}\T(R(4n))\T(K_{na}(4n)),
\end{equation}

\begin{equation}\label{eq11}
\M(\overline{AR}_{2n,4n-1}(\mathcal{A}_n))=2^{n}\T(R(4n-1))\T(K_{a}(4n-1)),
\end{equation}

\begin{equation}\label{eq12}
\M(\overline{AR}_{2n,4n-1}(\mathcal{B}_n))=2^{n}\T(R(4n-1))\T(K_{na}(4n-1)),
\end{equation}

where $\mathcal{A}_n=\{1,3,\dotsc,2n-1\}\cup\{2n+2,2n+4,\dotsc,4n\}$\\ and  $\mathcal{B}_n=\{2,4,\dotsc,2n\}\cup\{2n+1,2n+3,\dotsc,4n-1\}$.
\end{lemma}

\begin{figure}\centering
\includegraphics[width=10cm]{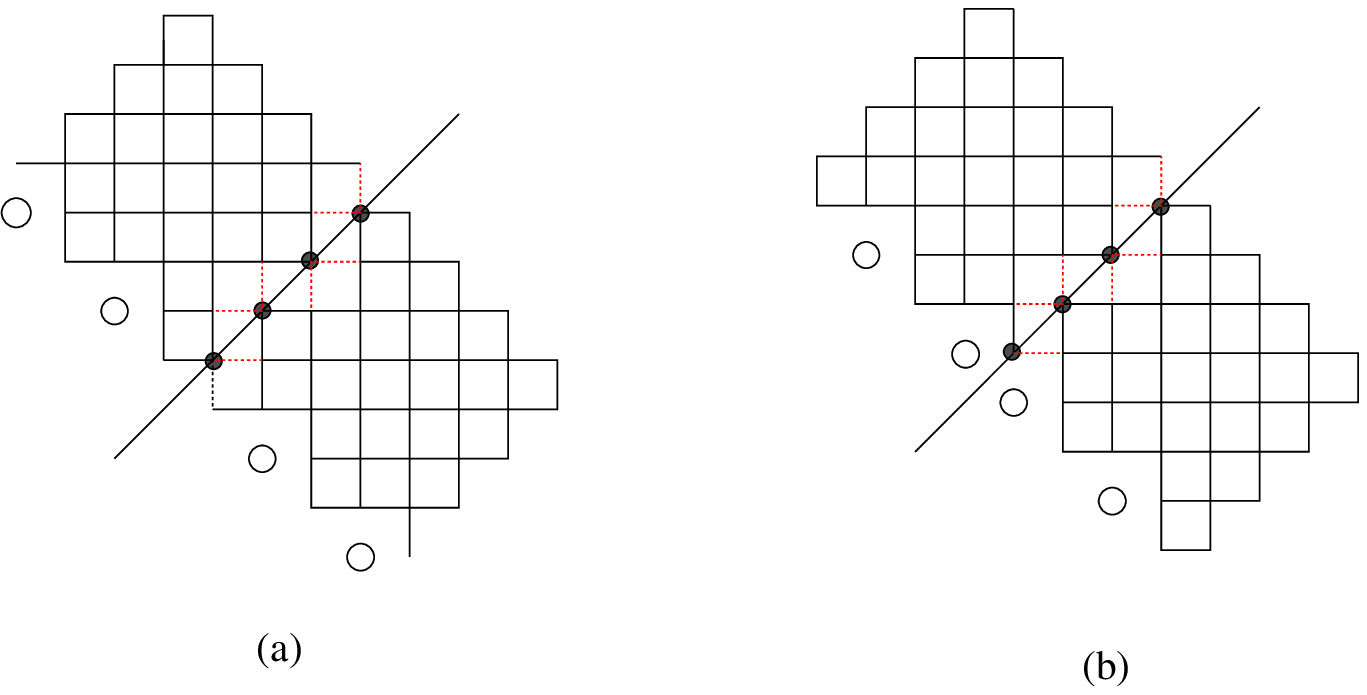}
\caption{Illustrating the proofs of (\ref{eq9}) and (\ref{eq10}) in Lemma \ref{lem3}.}
\label{QA8}
\end{figure}

\begin{figure}\centering
\includegraphics[width=10cm]{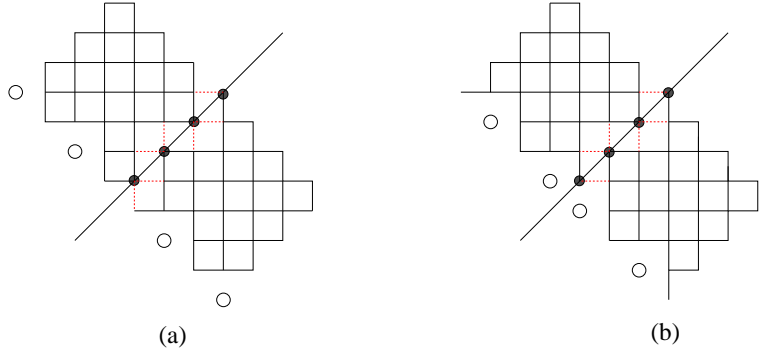}
\caption{Illustrating the proofs of (\ref{eq11}) and (\ref{eq12}) in Lemma \ref{lem3}.}
\label{QA7}
\end{figure}

\begin{proof}
Apply the Factorization Theorem to the graph $G:=AR_{2n,4n}(\mathcal{B}_n)$ with the symmetric axis $l$. There are $2n$ vertices of $G$ on $l$, so $w(G)=n$. It is easy to see that $G^+$ is isomorphic to the dual graph of $K_{a}(4n)$, and $G^-$ is isomorphic to the dual graph of $R(4n)$ (see Figure \ref{QA8}(a) for the case $n=2$). Then  (\ref{eq9}) follows.

Again, we apply the Factorization Theorem to the graph $\overline{G}:=AR_{2n,4n}(\mathcal{A}_n)$ with the symmetric axis $l'$.  It is easy to see that $\overline{G}^+$ is isomorphic to the dual graph of $R(4n)$, and $\overline{G}^-$ is isomorphic to the dual graph of $K_{na}(4n)$ (the case $n=2$ is shown in Figure \ref{QA8}(b)). Moreover,  it is easy to see $w(\overline{G})=n$. This implies  (\ref{eq10}).

Similarly, two equalities (\ref{eq11}) and (\ref{eq12}) can be obtained from applying the Factorization Theorem to $\overline{AR}_{2n,4n-1}(\mathcal{A}_n)$ and $\overline{AR}_{2n,4n-1}(\mathcal{B}_n)$; and the proofs are illustrated in Figures \ref{QA7}(a) and (b), respectively.
\end{proof}

Assume $\mathcal{A}_n$ and $\mathcal{B}_n$ are two sets defined in Lemma \ref{lem3}. Denote by
\[\Delta(\mathcal{A}_n):=\prod_{1\leq i< j\leq 2n}(a_j-a_i),\] the product is taken over all elements $a_i$'s in $\mathcal{A}_n$. Similarly we denote by $\Delta(\mathcal{B}_n)$ the corresponding product with elements in $\mathcal{B}_n$.

\begin{lemma}\label{lem6} For any $n\geq 1$
\begin{equation}\label{eq13}
\frac{\Delta(\mathcal{A}_n)}{\Delta(\mathcal{B}_n)}=\prod_{1\leq i,j\leq n}\frac{2n+1+2j-2i}{2n-1+2j-2i}
\end{equation}
\end{lemma}

\begin{proof}
We can partition $\mathcal{A}_n=\mathcal{C}_n\sqcup \mathcal{D}_n$, where $\mathcal{C}_n=\{1,3,\dotsc,2n-1\}$ and  where  \\ $\mathcal{D}_n=\{2n+2,2n+4,\dotsc,4n\}$. Therefore
\begin{align}
\Delta(\mathcal{A}_n)&=\prod_{ i<j \in \mathcal{C}_n}(j-i) \prod_{ i<j \in \mathcal{D}_n}(j-i) \prod_{ i\in \mathcal{C}_n; j\in \mathcal{D}_n}(j-i)\\
&=\prod_{1\leq i<j\leq n}((2j-1)-(2i-1)) \prod_{1\leq i<j\leq n}((2j+2n)-(2i+2n))\notag\\
&\quad \times \prod_{1\leq i,j\leq n}((2j+2n)-(2i-1))\\
&=\prod_{1\leq i<j\leq n}2(j-i) \prod_{1\leq i<j\leq n}2(j-i) \prod_{1\leq i,j\leq n}(2n+1+2j-2i)\\
&=2^{n(n-1)}\left(\prod_{1\leq i<j\leq n}(j-i)\right)^2\prod_{1\leq i,j\leq n}(2n+1+2j-2i)
\end{align}

Similarly, we have
\begin{equation}
\Delta(\mathcal{B}_n)
%&=\prod_{1\leq i<j\leq n}((2j)-(2i)) \prod_{1\leq i<j\leq n}((2j-1+2n)-(2i-1+2n))\\
%&\quad\times \prod_{1\leq i,j\leq n}((2j-1+2n)-2i)\\
%&=\prod_{1\leq i<j\leq n}2(j-i) \prod_{1\leq i<j\leq n}2(j-i) \prod_{1\leq i,j\leq n}(2n-1+2j-2i)\\
=2^{n(n-1)}\left(\prod_{1\leq i<j\leq n}(j-i)\right)^2\prod_{1\leq i,j\leq n}(2n-1+2j-2i)
\end{equation}
Then the equality (\ref{eq13}) follows.
\end{proof}

\begin{proof}[\textbf{Proof of Theorem \ref{main}}]
Since the dual graph $G$ of $R(n)$ is a bipartite graph, the numbers of vertices in two vertex classes of $G$ must be the same if $G$ admits perfect matchings. By enumerating vertices in each vertex class we can prove (\ref{main1}) (see example for $n=2$ in Figure \ref{main1eq}; the difference between the numbers of vertices in two classes are 1).

\begin{figure}\centering
\includegraphics[width=10cm]{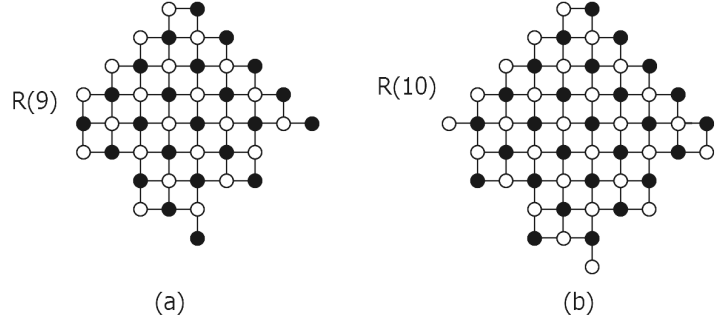}
\caption{The dual graphs of $R(9)$ and $R(10)$ with two vertex classes (black and white).}
\label{main1eq}
\end{figure}

Next, we prove four formulas (\ref{main2})-(\ref{main6}) by induction on $n\geq 1$.

\medskip
It is easy to verify those formulas for $n=1$. Assume that the formulas hold for some $n\geq1$, we will show that they hold also for $n+1$.

We have from Lemmas \ref{lem1} and \ref{lem1},  and induction hypothesis
\begin{align}\label{main3b}
\T(K_{a}(4n+4))&=\T(K_{a}(4n+2)) \text{\qquad\qquad\qquad\qquad\qquad (\textit{by Eq. (\ref{eq5}}))}\notag\\
&=2^{n+1}\T(K_{na}(4n+1)) \text{\qquad\qquad\qquad\quad\:(\textit{by Eq. (\ref{eq4}}))}\\
&=2^{2n+1}\T(K_{na}(4n)) \text{\qquad\qquad\qquad\quad\quad\;\:(\textit{by Eq. (\ref{eq2}}))}\\
&=2^{2n+1}2^{n(3n-1)/2}\prod_{1\leq i\leq j\leq n}\frac{2i+2j-1}{i+j-1} \text{\quad(\textit{by  Eq. (\ref{main5}) for $n$})}\\
&=2^{(3n^2+3n+2)/2}\prod_{1\leq i\leq j\leq n}\left(\frac{2i+2(j+1)-3}{i+(j+1)-1}\times \frac{i+(j+1)-1}{i+j-1}\right)\\
&=2^{(n+1)(3n+2)/2}2^{-n}\prod_{1\leq i\leq j\leq n}\frac{i+j}{i+j-1} \prod_{1\leq i< j\leq n+1}\frac{2i+2j-3}{i+j-1}\\
&=2^{(n+1)(3n+2)/2}2^{-n} \prod_{j=1}^{n}\frac{2j}{j} \prod_{1\leq i< j\leq n+1}\frac{2i+2j-3}{i+j-1}\\
&=2^{(n+1)(3(n+1)-1)/2}\prod_{1\leq i< j\leq n+1}\frac{2i+2j-3}{i+j-1}
\end{align}
It means that (\ref{main3}) holds for $n+1$.

From Lemmas  \ref{lem5}, \ref{lem3} and \ref{lem6}

\begin{align}\label{main4b}
\T(K_{a}(4n+5))&=\T(K_{a}(4n+3)) \text{\qquad\qquad\qquad\qquad\qquad\qquad\quad(\textit{by Eq. (\ref{eq6}}))}\notag\\
&=\T(K_{na}(4n+3))\frac{\T(K_{a}(4n+3))}{\T(K_{na}(4n+3))}\\
&=\T(K_{na}(4n+1))\frac{\T(K_{a}(4n+3))}{\T(K_{na}(4n+3))} \text{\;\quad\qquad\qquad(\textit{by Eq. (\ref{eq8}}))}\\
&=2^n\T(K_{na}(4n))\frac{\T(K_{a}(4n+3))}{\T(K_{na}(4n+3))} \text{\quad\quad\qquad\qquad(\textit{by Eq. (\ref{eq2}}))}\\
&=2^n\T(K_{na}(4n))\frac{\M(\overline{AR}_{2n+2,4n+3}(\mathcal{A}_{n+1}))}{\M(\overline{AR}_{2n+2,4n+3}(\mathcal{B}_{n+1}))} \text{\qquad\;\;(\textit{by Lemma \ref{lem3})}}\\
&=2^n\T(K_{na}(4n))\frac{\Delta(\mathcal{A}_{n+1})}{\Delta(\mathcal{B}_{n+1})} \text{\quad\qquad\qquad\quad\quad\quad\quad\textit{(by Lemma \ref{lem5}})}\\
&=2^n\T(K_{na}(4n))\prod_{1\leq i,j\leq n+1}\frac{2n+3+2j-2i}{2n+1+2j-2i} \text{\quad\quad(\textit{by Lemma \ref{lem6}})}\\
&=2^{n}\left(2^{n(3n-1)/2}\prod_{1\leq i\leq j \leq n}\frac{2i+2j-1}{i+j-1}\right)\prod_{1\leq i,j\leq n+1}\frac{2n+3+2j-2i}{2n+1+2j-2i} \\
 &=2^{n}\left(2^{n(3n-1)/2}\frac{\prod_{1\leq i\leq j \leq n+1}\frac{2i+2j-1}{i+j-1}} {\prod_{1\leq i \leq n+1}\frac{2i+2n+1}{i+n}}\right)\prod_{1\leq j\leq n+1}\frac{2n+1+2j}{2j-1}\\
&=2^{n}\left(2^{n(3n-1)/2}\prod_{1\leq i\leq j \leq n+1}\frac{2i+2j-1}{i+j-1}\right) \prod_{1\leq j\leq n+1}\frac{j+n}{2j-1}\\
&=2^{n}\left(2^{n(3n-1)/2}\prod_{1\leq i\leq j \leq n+1}\frac{2i+2j-1}{i+j-1}\right)\frac{(2n+1)!/n!}{(2n+1)!/(2^nn!)}\\
&=2^{(n+1)(3(n+1)-3)/2}\prod_{1\leq i\leq j \leq n+1}\frac{2i+2j-1}{i+j-1}
\end{align}
This implies that (\ref{main4}) holds for $n+1$.

\medskip

Similarly, we get the ratio $\dfrac{\T(K_{a}(4n+4))}{\T(K_{na}(4n+4))}$ by dividing  (\ref{eq9}) by (\ref{eq10}). Then from (\ref{eq7}) and the formula (\ref{main3}) for $n+1$, we can verify (\ref{main5}) for $n+1$. Again, two equalities (\ref{eq11}) and (\ref{eq12}) imply the ratio $\dfrac{\T(K_{a}(4n+3))}{\T(K_{na}(4n+3))}$; then from (\ref{eq8}) and the equality (\ref{main4}) for $n+1$, we get (\ref{main6})  for $n+1$.

Finally, from (\ref{eq1}) and (\ref{eq9}) together with the equality (\ref{main3}) for $n+1$, we can prove (\ref{main2}) for $n+1$.
\end{proof}

\subsection*{Acknowledgements}
Thanks to Professor Mihai Ciucu and Professor James Propp for giving me the manuscript of \cite{JP}.

\end{document}